\documentclass[reqno]{amsart}
\usepackage{amscd}
\usepackage{amssymb}
\usepackage[dvips]{graphics}
\usepackage[margin=1in,includeheadfoot]{geometry}
\usepackage{url}
\usepackage{hyperref}
\usepackage{cite}
\usepackage{color}

\def\beq{\begin{equation}}
\def\eeq{\end{equation}}
\def\ba{\begin{array}}
\def\ea{\end{array}}

\def\R{\mathbb R}

\newtheorem{thm}{Theorem}[section]
\newtheorem{lm}[thm]{Lemma}

\theoremstyle{definition}
\newtheorem{rem}[thm]{Remark}

\newtheorem{df}[thm]{Definition}

\newtheorem{assump}{Assumption}

\theoremstyle{remark}

\begin{document}
\pagestyle{plain}\today
\title{Interior pointwise $C^{1}$ and $C^{1,1}$ regularity of solutions for general semilinear elliptic equation in nondivergence form}

\author{\small{Jingqi Liang}\\
 \small{Institute of Natural Sciences, Shanghai Jiao Tong University}\\
 \small{Shanghai, China}}

\begin{abstract}
In this paper, we obtain $C^{1}$ and $C^{1,1}$ regularity of $L^{n}$-viscosity solutions for general semilinear elliptic equation in nondivergence form under some more weaker assumptions, which generalize the result for equations with nonhomogeneous term $f(x)$ to $f(x,u)$. In particular, the nonhomogeneous term $f(x,u)$ is assumed optimally to satisfy unform Dini continuity condition in $u$ and modified $C^{1,1}$ Newtonian potential condition in $x$. For unbounded coefficients, if $a_{ij}$ is $C_{n}^{-1,1}$ at $x_{0}\in\Omega$ with small modulus, $b_{i}\in L^{q}(\Omega)$ for some $q>n$, the solution is $C^{1}$ at $x_{0}$. Furthermore, if $a_{ij},~b_{i}$ are Dini continuous at $x_{0}$, the solution is $C^{1,1}$ at $x_{0}$.
\end{abstract}

\maketitle


\section{Introduction}
In this paper, we study the pointwise interior $C^{1}$ and $C^{1,1}$ regularity of $L^{n}$-viscosity solutions for the semilinear elliptic equation in nondivergence form. Suppose that $u$ is a solution of
\begin{equation}\label{eq1}
a_{ij}D_{ij}u+b_{i}D_{i}u=f(x,u)\quad\text{in}~~\Omega,
\end{equation}
where $\Omega$ is a bounded domain and $B_{2}\Subset\Omega$, the coefficients $a_{ij},~b_{i}(1\leq i,j\leq n)$ are measurable functions on $\Omega$,  $(a_{ij}(x))_{n\times n}$ is symmetric and satisfies the uniform ellipticity condition with the ellipticity constant $0<\Lambda\leq1$:
\begin{equation}\label{aij}
\Lambda|\xi|^{2}\leq \sum\limits_{ij}a_{ij}(x)\xi_{i}\xi_{j}\leq\Lambda^{-1}|\xi|^{2},~~\forall x\in \Omega,~~\xi\in \mathbb{R}^{n},
\end{equation}
and $b_{i}\in L^{q}(\Omega)(q>n)$ with
  \begin{equation}\label{b}
  \sum\limits_{i}\|b_{i}\|_{L^{q}(\Omega)}\leq\Lambda_{1},
  \end{equation}
for some constant $\Lambda_{1}$.

In the past two decades, many scholars studied $C^{1,1}$ regularity of solutions for the following semilinear elliptic equation:
\begin{equation}\label{semi}
\Delta u=f(x,u),\quad x\in B_{1}.
\end{equation}
Especially, if $f(x,u)=\chi_{\{u\neq0\}}$ where $\chi_{\{u\neq0\}}$ is the characteristic function of the set $\{u\neq0\}$, $(\ref{semi})$ is classical obstacle problem. More generally, the case $f(x,u)=f(x)\chi_{\{u\neq0\}}$ has been considered. For nonnegative solutions, it is classical that $u\in C^{1,1}(B_{\frac{1}{2}})$ if $f\in C^{\text{Dini}}(B_{1})$, see \cite{BL2001,Ca1998}. For the no-sign solution, Petrosyan and Shahgholian \cite{PS2007} proved $C^{1,1}$ regularity of $u$ when $f$ satisfies a double Dini condition. Later, Andersson, Lindgren and Shahgholian gave the weakest assumption to get the interior $C^{1,1}$ regularity, that is $f\ast N\in C^{1,1}(B_{1})$, where $N$ is the Newtonian potential. In addition, they proposed a natural question: whether $C^{1,1}$ regularity holds for $D_{j}(a_{ij}D_{i}u)=f(x)\chi_{\{u\neq0\}}$ under proper assumption on $a_{ij}$?

There are also some regularity results for the general case. Shahgholian \cite{6} showed the interior $C^{1,1}$ regularity of $u$ under the conditions that $f(x,u)$ is Lipschitz continuous in $x$, uniformly in $u$, and $\partial_{u}f\geq-C$. Koch and Nadirashvili \cite{KN15} gave an example to illustrate that the continuity of $f$ is not sufficient to deduce that weak solution of $\Delta u=f(u)$ is $C^{1,1}$.

Based on above results, in 2017, Indrei, Minne and Nurbekyan \cite{1} proved the interior $C^{1,1}$ regularity of weak solution $u$ for $(\ref{semi})$ under the almost optimal assumptions that $f(x,u)$ satisfies the uniform Dini continuity condition in $u$ and the uniform $C^{1,1}$ Newtonian potential condition in $x$. Inspired by this work, we considered in \cite{LWZ2023,LWZ2022} the pointwise Lipschitz regularity up to the boundary for the following semilinear elliptic equation in divergence form:
$$-D_{j}(a_{ij}D_{i}u)+b_{i}D_{i}u=-\text {div}\overrightarrow{\mathbf{F}}(x,u)\qquad\text{in}~~\Omega,
$$
with boundary value $u=g$ on $\partial\Omega$. In particular, in \cite{LWZ2023}, we began with the simple case: $a_{ij}=\delta_{ij},~b_{i}=0$, where $\{\delta_{ij}\}_{n\times n}$ is identity matrix. We gave an assumption on $\overrightarrow{\mathbf{F}}(x,u)$ which is parallel to Indrei's: $\overrightarrow{\mathbf{F}}(x,u)$ satisfies uniform Dini continuity condition in $u$ and uniform Lipschitz Newtonian potential condition in $u$. We established that, if $\partial\Omega$ is $C^{1,\text{Dini}}$ or Reifenberg $C^{1,\text{Dini}}$ at some boundary point, the solution is Lipschitz continuous at this point. Furthermore, if $a_{ij}$ satisfies Dini decay condition and $b_{i}\in L^{q}$ for some $q>n$, the same result also holds for the general equation, see \cite{LWZ2022}.

The above researches focus on the $C^{1}$ and $C^{1,1}$ regularity of weak solutions for semilinear equations in divergence form. But how about the regularity of strong solutions or viscosity solutions for general semilinear equation in nondivergence form $(\ref{eq1})$? Even interior regularity is unknown. This paper will answer this question after assuming suitable conditions on $a_{ij},~b_{i}$ and $f(x,u)$. For nonhomogeneous term $f(x,u)$, we reserve the uniform Dini continuity condition in $u$ and modify the $C^{1,1}$ Newtonian potential condition in $x$ slightly, i.e. we give the  following assumptions:

\begin{assump}\label{as1} Assume $f(x,t)\in L^{\infty}(\Omega\times \R)$. Moreover
$f(x,t)$ is Dini continuous in $t$ with continuity modulus $\varphi(r)$, uniformly in $x$, i.e.
$$\left|f(x,t_{2})-f(x,t_{1})\right|\leq\varphi(|t_{2}-t_{1}|),
$$
and $\displaystyle\int_{0}^{t_{0}} \displaystyle\frac{\varphi(t)}{t} dt<\infty$ for some $t_{0}>0.$
\end{assump}
\begin{assump}\label{as2}
For each $x_{0}\in B_{1}$ and each $t\in \R$, there exists a function $v_{x_{0}}(\cdot,t)$ in $B_{1}(x_{0})$ satisfying
$$
a_{ij}(x_{0})D_{ij}v_{x_{0}}(\cdot, t)=f(\cdot,t) \quad \text{in}~~B_{1}(x_{0}).
$$
Furthermore, $v_{x_{0}}(\cdot,t)$  is a $C^{1,1}$ function which is uniform in $x_0$ and $t$, that is there exists a constant $T$ such that for any $x_{0}\in B_{1}$, $a,b\in\mathbb{R}$,
\begin{eqnarray*}
\sup_{a\leq t\leq b}\|D^{2}v_{x_{0}}(\cdot, t)\|_{L^{\infty}(B_{1}(x_{0}))}\leq T.
\end{eqnarray*}
\end{assump}
\begin{rem}
In the sequel, for each $x_{0}\in B_{1}$, we let $v_{x_{0}}(x)$ solve
$$a_{ij}(x_{0})D_{ij}v_{x_{0}}(x)=f(x,u(x_{0})) \quad \text{in}~~B_{1}(x_{0}).$$
\end{rem}

%
%
%
%
%

Next we need to propose some weaker assumptions on all $a_{ij}$ and $b_{i}$. For the linear equation only with the leading terms
$$a_{ij}D_{ij}u=f(x)\quad \text{in}~~\Omega,
$$
it is well-known that Caffarelli in \cite{Ca} proved that if $a_{ij}$ are uniformly close to constant and
$f$ has controlled growth, u is a solution of class $C^{1,\alpha}$, and  if $a_{ij}$ are $C^{\alpha}$ and $f$ is $C^{\alpha}$, then $u$ belongs to $C_{\text{loc}}^{2,\alpha}$. The similar results were generalized to fully nonlinear equations in 1989 \cite{Ca1989}. In 2017, Dong and Kim \cite{DK2017} proved the interior $C^{2}$ regularity of the $W^{2,2}$ strong solution for above equation by Campanato method under the assumptions that $a_{ij}$ and $f(x)$ have $L^{1}$ Dini mean oscillation. Moreover, for general equation with lower order term
\begin{equation}\label{aijge}
a_{ij}D_{ij}u+b_{i}D_{i}u=f(x)\quad \text{in}~~\Omega,
\end{equation}
if all $b_{i}$ have Dini mean oscillation as well, the same result still holds. The global $C^{2}$ regularity was also obtained by supposing $\partial\Omega\in C^{2,\text{Dini}}$ for zero-value Dirichlet problem in \cite{DEK2018}. In 2019, Huang, Zhai and Zhou \cite{HZZ2019} considered the boundary $C^{1}$ regularity of $W_{\text{loc}}^{2,n}(\Omega)\cap C(\overline{\Omega})$ strong solution for $(\ref{aijge})$ with zero boundary value. They proved that if $\partial\Omega$ is $C^{1,\text{Dini}}$, $b_{i}$ and $f$ are $C_{n}^{-1,\text{Dini}}$ at boundary point $x_{0}$, then $u$ is differentiable at $x_{0}$. Later Lian, Wang and Zhang \cite{LWZarxiv} focused  on the $L^{n}$-viscosity solution of fully nonlinear elliptic equation and gave an comprehensive regularity result for $(\ref{aijge})$ even up to the boundary, including pointwise $C^{k,\alpha}$, $C^{k}$ regularity and so on. They showed that if $a_{ij}$ is $C_{n}^{-1,1}$ at $x_{0}$ with small modulus, $b_{i}\in L^{q}$ where $q>n$, $f\in C_{n}^{-1,\text{Dini}}(x_{0})$ for $x_{0}\in\Omega$, then the viscosity solution is $C^{1}$ at $x_{0}$. If $a_{ij},~b_{i},~f\in C^{\text{Dini}}(x_{0})$, then $u\in C^{2}(x_{0})$. For more results about interior and boundary $C^{1}$ and $C^{2}$ regularity of solutions for the equations with nonhomogeneous term $f(x)$, we refer readers to \cite{Bu,4,5,K1997,7,13,14,15,18,16,3,2}.

In this paper, we continue to study the regularity of solutions under some more weaker assumptions. Let us  begin with some definitions.

\begin{df}\label{as3-1} Let $x_{0}\in B_{1}$, we say $a_{ij}\in L^{n}(\Omega)$ is $C_{n}^{-1,1}$ at $x_{0}$ with small modulus $\nu$, if there exists a positive constant $r_{0}\leq1$ and a sufficiently small constant $\nu$ such that for any $0<r\leq r_{0}$,
$$\|a_{ij}(x)-a_{ij}(x_{0})\|_{L^{n}(B_{r}(x_{0}))}\leq\nu r.
$$
We denote $\|a_{ij}\|_{C_{n}^{-1,1}(x_{0})}=\nu$.
\end{df}

\begin{df}\label{as3} Let $x_{0}\in B_{1}$, we say $a_{ij}\in L^{\infty}(\Omega)$ is $C^{\text{Dini}}$ at $x_{0}$, if there exists a positive constant $r_{0}\leq1$ and a Dini modulus of continuity $\omega_{1}(r)$ satisfying $\displaystyle\int_{0}^{r_{0}} \displaystyle\frac{\omega_{1}(r)}{r} d r<\infty$ such that for any $0<r\leq r_{0}$,
    $$\|a_{ij}(x)-a_{ij}(x_{0})\|_{L^{\infty}(B_{r}(x_{0}))}\leq\omega_{1}(r).$$
\end{df}


\begin{df}\label{as5} Let $x_{0}\in B_{1}$, we say $b_{i}\in L^{\infty}(\Omega)$ is $C^{\text{Dini}}$ at $x_{0}$, if there exists a positive constant $r_{0}\leq1$ and a Dini modulus of continuity $\omega_{2}(r)$ satisfying $\displaystyle\int_{0}^{r_{0}}\displaystyle\frac{\omega_{2}(r)}{r} d r<\infty$ such that for any $0<r\leq r_{0}$,
    $$\|b_{i}(x)-b_{i}(x_{0})\|_{L^{\infty}(B_{r}(x_{0}))}\leq\omega_{2}(r).$$
\end{df}

\begin{df}
Let $x_{0}\in B_{1}$, $f:\Omega\rightarrow\mathbb{R}$ be a function. We say that $f$ is $C^{k}(k\geq0)$ at $x_{0}\in B_{1}$ if there exists a $k$-th order polynomial $P_{x_{0}}$ such that
\begin{eqnarray*}
\sup_{|x-x_{0}|\leq r}\frac{|f(x)-P_{x_{0}}(x)|}{|x-x_{0}|^{k}}\rightarrow0~\text{as}~r\rightarrow0.
\end{eqnarray*}
\end{df}
 \begin{rem}
Any modulus of continuity $\omega(t)$ is non-decreasing, subadditive, continuous and satisfies $\omega(0)=0$. Hence any modulus of continuity $\omega(t)$ satisfies
\begin{equation}\label{diniproperty}\frac{\omega(r)}{r}\leq2\frac{\omega(h)}{h},\quad 0<h<r.
\end{equation}
For more details about the properties of modulus of continuity, we refer readers to \cite{K1997,7,2}.
\end{rem}

Our main Theorems are following:
\begin{thm}\label{mr1}Assume that $x_{0}\in B_{1}$, $f(x,u)$ satisfies Assumption \ref{as1} and \ref{as2}, and $a_{ij}$ is $C_{n}^{-1,1}$ at $x_{0}$ with small modulus $\nu$. Then the bounded solution of $(\ref{eq1})$, $(\ref{aij})$ and $(\ref{b})$ is $C^{1}$ at $x_{0}$.
\end{thm}
\begin{thm}\label{mr2}Assume that $x_{0}\in B_{1}$, $f(x,u)$ satisfies Assumption \ref{as1} and \ref{as2}, $a_{ij}$ is $C^{\text{Dini}}$ at $x_{0}$ with Dini modulus of continuity $\omega_{1}(r)$, and $b_{i}\in L^{\infty}(\Omega)$ is $C^{\text{Dini}}$ at $x_{0}$ with Dini modulus of continuity $\omega_{2}(r)$. Then the bounded solution of $(\ref{eq1})$, $(\ref{aij})$ is $C^{1,1}$ at $x_{0}$.
\end{thm}

Theorem $\ref{mr1}$ and $\ref{mr2}$ are pointwise results, and it is sufficient to show the $C^{1}$ and $C^{1,1}$ regularity at 0 in the following.  The classical interior $C^{1}$ and $C^{1,1}$ regularity can be derived straightly from the pointwise regularities. To prove above two theorems, we will use perturbation technique and iteration method proposed in \cite{Ca,Ca1989} and generalized by Wang in \cite{W}. These two methods were widely used to prove the pointwise regularity of solutions in the last two decades, such as in \cite{4,5,HZZ2019,13,14,15,LWZarxiv,18,LWZ2023,LWZ2022,16,3}. The main idea is to approximate the solution by linear functions and second order polynomials respectively in different scales, and the errors can be proved convergent. The key step is to find a function solving the equations only with constant leading terms to approximate the solution on a small scale, then linear function or second order polynomial can be found naturally. Based on this, we can iterate step by step.

It is worthy to mention that our method is different from that which was used to prove the local $C^{1,1}$ regularity in \cite{ALS,1}. They constructed a space of Hessian generated by second order homogeneous harmonic polynomials on balls with radius $r>0$, then considered the $L^{2}$ projection of $D^{2}u$ on this space and showed that the projections stay uniformly bounded as $r\rightarrow0$.

The paper is organized as follows. In Section 2, we will recall the definition of $L^{n}$-viscosity solution and give some preliminary results. Our main efforts are to get the approximation lemma and key lemma. Then we will give the proof of $C^{1}$ and $C^{1,1}$ regularity in Section 3 and 4 respectively.

\section{Approximation Lemma}
In this section, the approximation lemma will be proved. The approximation lemmas in \cite{LWZ2023,LWZ2022} are to find harmonic functions to approximate solutions. Here is different slightly. The main idea is to find a function $h$ which satisfies a homogeneous equation only with constant coefficients: $a_{ij}(0)D_{ij}h=0$, to approximate the solution. Then linear functions and second order polynomials can be found naturally by taking the first order and second order Taylor polynomials of $h$ at $0$ respectively. The A-B-P maximum principle and interior H\"{o}lder estimates for viscosity solution will be used here. We firstly recall the definition of $L^{n}$-viscosity solution and some preliminary results.
\begin{df}[$L^{p}$-viscosity solution]
We say that $u\in C(\Omega)$ is an $L^{p}$-viscosity subsolution (respectively, supersolution) of $(\ref{eq1})$ if whenever $\phi\in W^{2,p}_{\text{loc}}(\Omega)$, $\varepsilon>0$ and $\Omega'\subset \Omega$ open are such that
$$a_{ij}(x)D_{ij}\phi(x)+b_{i}(x)D_{i}\phi(x)-f(x,u(x))\leq-\varepsilon
$$
$$
(\text{resp.}~~ a_{ij}(x)D_{ij}\phi(x)+b_{i}(x)D_{i}\phi(x)-f(x,u(x))\geq\varepsilon)
$$
for a.e. $x\in\Omega'$, then $u-\phi$ cannot have a local maximum (minimum) in $\Omega'$.

We call $u\in C(\Omega)$ an $L^{p}$-viscosity solution of $(\ref{eq1})$ if it is both an $L^{p}$-viscosity subsolution and supersolution of $(\ref{eq1})$.
\end{df}
\begin{rem}
In this paper, we consider the $L^{n}$-viscosity solution and we may write ``viscosity solution" for short. We refer readers to \cite{CCKS1996} for more about the $L^{p}$-viscosity solution.
\end{rem}

The following two theorems can be found in \cite{N2019}, the A-B-P maximum principle and H\"{o}lder regularity of viscosity solutions for general linear elliptic equation $(\ref{aijge})$.
\begin{thm}[A-B-P maximum principle]
Let $\Omega$ bounded, $b_{i}\in L^{q}(\Omega)$ and $f\in L^{p}(\Omega)$, for $q\geq p\geq n$, $q>n$. If $u\in C(\overline{\Omega})$ is an $L^{p}$ solution of $(\ref{aij})$, $(\ref{b})$ and $(\ref{aijge})$, then the following estimate holds:
$$\|u\|_{L^{\infty}(\Omega)}\leq\|u\|_{L^{\infty}(\partial\Omega)}+C\|f\|_{L^{p}(\Omega)},
$$
where $C$ is a constant depending on $n,p,\Lambda,\Lambda_{1},\text{diam}(\Omega)$.
\end{thm}
\begin{thm}[$C^{\alpha}$ regularity]\label{visholder}
Let $u\in C(\Omega)$ be an $L^{p}$-viscosity solution of $(\ref{aij})$, $(\ref{b})$ and $(\ref{aijge})$ with $f\in L^{p}(\Omega)$, for $q\geq p\geq n$, $q>n$. Then there exists $\alpha\in(0,1)$ depending on $n,p,\Lambda,\Lambda_{1}$ such that $u\in C^{\alpha}_{\text{loc}}(\Omega)$ and for any $\Omega'\Subset\Omega$,
$$\|u\|_{C^{\alpha}(\Omega')}\leq K_{1}(\|u\|_{L^{\infty}(\Omega)}+\|f\|_{L^{p}(\Omega)}),
$$
where $K_{1}$ depends on $n,p,\Lambda,\Lambda_{1},\text{dist}(\Omega',\Omega)$.

In addition, if $u\in C(\overline{\Omega})\cap C^{\beta}(\partial\Omega)$ and $\Omega$ satisfies a uniform exterior cone condition with size $L$, then there exists $\alpha_{0}=\alpha_{0}(n,p,\Lambda,\Lambda_{1},L)\in(0,1)$ and $\alpha=\min\{\alpha_{0},\displaystyle\frac{\beta}{2}\}$ such that
$$\|u\|_{C^{\alpha}(\overline{\Omega})}\leq K_{1}(\|u\|_{L^{\infty}(\Omega)}+\|f\|_{L^{p}(\Omega)}+\|u\|_{C^{\beta}(\partial\Omega)}),
$$
where $K_{1}$ depends on $n,p,\Lambda,\Lambda_{1},L,\text{diam}(\Omega)$.
\end{thm}

\begin{lm}[Approximation Lemma]\label{fs}
Assume $u$ is a solution of
\begin{eqnarray*}
a_{ij}D_{ij}u+b_{i}D_{i}u=f\qquad\text{in}~~B_{1},
\end{eqnarray*}
with $f(x)\in L^{n}(B_{1})$. If $a_{ij}$ and $b_{i}$ satisfy $(\ref{aij}),~(\ref{b})$, and
$$\|a_{ij}-a_{ij}(0)\|_{L^{n}(B_{1})}\leq \varepsilon_{1}, \quad \|b_{i}\|_{L^{q}(B_{1})}\leq \varepsilon_{2},
$$
for some $\varepsilon_{1},~\varepsilon_{2}>0$ small enough, then for some $\alpha>0$, there exists a function $h$ defined in $B_{\frac{3}{4}}$ satisfying
\begin{eqnarray*}
\left\{
\begin{array}{rcll}
a_{ij}(0)D_{ij}h&=&0\qquad&\text{in}~~B_{\frac{3}{4}},\\
h&=&u\qquad&\text{on}~~\partial B_{\frac{3}{4}}, \\
\end{array}
\right.
\end{eqnarray*}
such that
\begin{eqnarray*}
\left\|u-h\right\|_{L^{\infty}(B_{\frac{1}{2}})}
\leq C((\varepsilon_{1}+\varepsilon_{2})^{\alpha}\|u\|_{L^{\infty}(B_{1})}+\|f\|_{L^{n}(B_{1})}),
\end{eqnarray*}
where $\alpha=\alpha(n,\Lambda,\Lambda_{1})>0$, $C=C(n,\Lambda,\Lambda_{1})$.
\end{lm}
\begin{proof}
Firstly, by Theorem \ref{visholder}, there exists $0<\beta_{0}<1$ such that $u\in C^{\beta_{0}}(\overline{B_{\frac{3}{4}}})$,
\begin{eqnarray*}
\|u\|_{C^{\beta_{0}}(\overline{B_{\frac{3}{4}}})}\leq C\left(\|u\|_{L^{\infty}(B_{1})}+\|f\|_{L^{n}(B_{1})}\right),
\end{eqnarray*}
where $C$ depends on $n,\Lambda,\Lambda_{1}$. We construct $h$ by solving the problem
\begin{eqnarray*}
\left\{
\begin{array}{rcll}
a_{ij}(0)D_{ij}h&=&0\qquad&\text{in}~~B_{\frac{3}{4}},\\
h&=&u\qquad&\text{on}~~\partial B_{\frac{3}{4}}. \\
\end{array}
\right.
\end{eqnarray*}
By A-B-P maximum principle, we have
\begin{eqnarray*}
\|h\|_{L^{\infty}(B_{\frac{3}{4}})}\leq \|u\|_{L^{\infty}(\partial B_{\frac{3}{4}})}\leq \|u\|_{L^{\infty}(B_{1})}.
\end{eqnarray*}
Using Theorem \ref{visholder} again, it follows that there exists $\beta\leq\displaystyle\frac{\beta_{0}}{2}$ such that $h\in C^{\beta}(\overline{B_{\frac{3}{4}}})$ and
\begin{eqnarray*}
\|h\|_{C^{\beta}(\overline{B_{\frac{3}{4}}})}\leq C(\|h\|_{L^{\infty}(B_{\frac{3}{4}})}+\|u\|_{C^{\beta_{0}}(\partial B_{\frac{3}{4}})})\leq C\left(\|u\|_{L^{\infty}(B_{1})}+\|f\|_{L^{n}(B_{1})}\right).
\end{eqnarray*}
For $\delta>0$ small enough and any $x\in \partial B_{\frac{3}{4}-\delta}$, there exists $x_{0}\in \partial B_{\frac{3}{4}}$ such that $|x-x_{0}|=\delta$. It follows that
\begin{eqnarray*}
\frac{|u(x)-h(x)|}{\delta^{\beta}}&=&\frac{|u(x)-h(x)-(u(x_{0})-h(x_{0}))|}{|x-x_{0}|^{\beta}}\\
&\leq&\|u\|_{C^{\beta}(\overline{B_{\frac{3}{4}}})}+\|h\|_{C^{\beta}(\overline{B_{\frac{3}{4}}})}\\
&\leq& C\left(\|u\|_{L^{\infty}(B_{1})}+\|f\|_{L^{n}(B_{1})}\right).
\end{eqnarray*}
Therefore
$$\|u-h\|_{L^{\infty}(\partial B_{\frac{3}{4}-\delta})}\leq C\left(\|u\|_{L^{\infty}(B_{1})}+\|f\|_{L^{n}(B_{1})}\right)\delta^{\beta}.
$$
Using the property of solutions for equations with constant coefficients, we know that $h$ is smooth in $B_{\frac{3}{4}}$ with
$$\|Dh\|_{L^{\infty}(B_{\frac{3}{4}-\delta})}\leq C\|h\|_{L^{\infty}(B_{\frac{3}{4}})}\delta^{-1}\leq C\|u\|_{L^{\infty}(B_{1})}\delta^{-1},
$$
$$\|D^{2}h\|_{L^{\infty}(B_{\frac{3}{4}-\delta})}\leq C\|h\|_{L^{\infty}(B_{\frac{3}{4}})}\delta^{-2}\leq C\|u\|_{L^{\infty}(B_{1})}\delta^{-2}.
$$
Since $u-h$ solves the following Dirichlet problem,
\begin{eqnarray*}
\left\{
\begin{array}{rcll}
a_{ij}D_{ij}(u-h)+b_{i}D_{i}(u-h)&=&f+(a_{ij}(0)-a_{ij})D_{ij}h-b_{i}D_{i}h\qquad&\text{in}~~B_{\frac{3}{4}-\delta},\\
u-h&=&u-h\qquad&\text{on}~~\partial B_{\frac{3}{4}-\delta}, \\
\end{array}
\right.
\end{eqnarray*}
then by A-B-P maximum principle and taking $\delta^{2+\beta}=\varepsilon_{1}+\varepsilon_{2}$, $\alpha=\displaystyle\frac{\beta}{2+\beta}$, we have
\begin{eqnarray*}
\|u-h\|_{L^{\infty}(B_{\frac{3}{4}-\delta})}&\leq&\|u-h\|_{L^{\infty}(\partial B_{\frac{3}{4}-\delta})}+C\|f\|_{L^{n}(B_{\frac{3}{4}-\delta})}\\
&&+C\left(\|(a_{ij}(0)-a_{ij})D_{ij}h\|_{L^{n}(B_{\frac{3}{4}-\delta})}+\|b_{i}D_{i}h\|_{L^{n}(B_{\frac{3}{4}-\delta})}   \right)\\
&\leq&C\left(\|u\|_{L^{\infty}(B_{1})}+\|f\|_{L^{n}(B_{1})}\right)\delta^{\beta}+C\|f\|_{L^{n}(B_{1})}\\
&&+C\left(\|u\|_{L^{\infty}(B_{1})}+\|f\|_{L^{n}(B_{1})}\right)(\delta^{-1}\varepsilon_{1}+\delta^{-2}\varepsilon_{2})\\
&\leq&C((\varepsilon_{1}+\varepsilon_{2})^{\alpha}\|u\|_{L^{\infty}(B_{1})}+\|f\|_{L^{n}(B_{1})}).
\end{eqnarray*}
Thus we get the desired result.
\end{proof}
\begin{lm}[Key lemma]\label{kl}
Let $a_{ij}$ and $b_{i}$ satisfy the assumptions in Lemma $\ref{fs}$, then there exists $\alpha>0$, $0<\lambda<1$ and universal constants $C_{0}, C_{1}, C_{2}>0$ such that for any functions $f(x,u)\in L^{\infty}(B_{1}\times \mathbb{R})$, if $u$ is the solution of
\begin{eqnarray*}
a_{ij}D_{ij}u+b_{i}D_{i}u=f(x,u)\qquad&\text{in}~~B_{1},
\end{eqnarray*}
and $v$ is a $C^{1,1}$ solution of
$$
a_{ij}(0)D_{ij}v=f(x,u(0))\qquad \text{in}~~B_{1}
$$
with $\|D^{2}v\|_{L^{\infty}(B_{1})}\leq T_{1}$ and $\|Dv\|_{L^{\infty}(B_{1})}\leq T_{2}$, then there exists a linear function $L(x)=A+B\cdot x$ such that
\begin{eqnarray*}
\|u-v-L(x)\|_{L^{\infty}(B_{\lambda})}&\leq& C_{1}(\lambda^{2}+(\varepsilon_{1}+
\varepsilon_{2})^{\alpha})\|u-v\|_{L^{\infty}(B_{1})}\\
&&+C_{2}\left(\|f(x,u)-f(x,u(0))\|_{L^{\infty}(B_{1})}+T_{1}\varepsilon_{1}+T_{2}\varepsilon_{2}\right),
\end{eqnarray*}
and
$$0<|A|+|B|\leq C_{0}\|u-v\|_{L^{\infty}(B_{1})}.
$$
\end{lm}
\begin{proof}
By the definition of $u$ and $v$ we get
\begin{eqnarray*}
a_{ij}D_{ij}(u-v)+b_{i}D_{i}(u-v)=f(x,u)-f(x,u(0))+(a_{ij}(0)-a_{ij})D_{ij}v-b_{i}D_{i}v\qquad\text{in}~~B_{1}.
\end{eqnarray*}
Then by Lemma $\ref{fs}$, for some $\alpha>0$, there exists a function $h$ defined in $B_{\frac{3}{4}}$ satisfying
\begin{eqnarray*}
\left\{
\begin{array}{rcll}
a_{ij}(0)D_{ij}h&=&0\qquad&\text{in}~~B_{\frac{3}{4}},\\
h&=&u-v\qquad&\text{on}~~\partial B_{\frac{3}{4}}, \\
\end{array}
\right.
\end{eqnarray*}
such that
\begin{equation}\label{11}
\begin{aligned}
\left\|u-v-h\right\|_{L^{\infty}(B_{\frac{1}{2}})}
&\leq C(\|f(x,u)-f(x,u(0))+(a_{ij}(0)-a_{ij})D_{ij}v-b_{i}D_{i}v\|_{L^{n}(B_{1})})\\
&\quad+C(\varepsilon_{1}+\varepsilon_{2})^{\alpha}\|u-v\|_{L^{\infty}(B_{1})}\\
&\leq C(\|f(x,u)-f(x,u(0))\|_{L^{\infty}(B_{1})}+T_{1}\varepsilon_{1}+T_{2}\varepsilon_{2})\\
&\quad+C(\varepsilon_{1}+\varepsilon_{2})^{\alpha}\|u-v\|_{L^{\infty}(B_{1})},
\end{aligned}
\end{equation}
where we have used $\|D^{2}v\|_{L^{\infty}(B_{1})}\leq T_{1}$, $\|Dv\|_{L^{\infty}(B_{1})}\leq T_{2}$ and $\|a_{ij}-a_{ij}(0)\|_{L^{n}(B_{1})}\leq \varepsilon_{1},~ \|b_{i}\|_{L^{q}(B_{1})}\leq \varepsilon_{2}$.

Take $L$ be the first order Taylor polynomial of $h$ at 0, i.e. $L(x)=Dh(0)\cdot x+h(0).$ Then there exists $\xi\in B_{\frac{1}{4}}$ such that for $|x|\leq\displaystyle\frac{1}{4}$,
\begin{equation}\label{12}
|h(x)-L(x)|\leq\frac{1}{2}|D^{2}h(\xi)||x|^{2}.
\end{equation}
Note that $h$ satisfies
$$|D^{2}h(x)|+|Dh(x)|+|h(x)|\leq C_{0}\|h\|_{L^{\infty}(B_{\frac{1}{2}})}\leq C_{0}\|u-v\|_{L^{\infty}(B_{1})},\quad \text{for}~|x|\leq\frac{1}{4},
$$
where $C_{0}$ is a constant depending only on $n,\Lambda,\Lambda_{1}$. It follows that
$$|D^{2}h(\xi)|+|Dh(0)|+|h(0)|\leq C_{0}\|u-v\|_{L^{\infty}(B_{1})}.
$$
Finally, combining (\ref{11}) with (\ref{12}), if we take $0<\lambda<\displaystyle\frac{1}{4}$, set $A=h(0)$, $B=Dh(0)$, then we have
\begin{eqnarray*}
\|u-v-L(x)\|_{L^{\infty}(B_{\lambda})}&\leq&\|u-v-h\|_{L^{\infty}(B_{\lambda})}+\|h-L\|_{L^{\infty}(B_{\lambda})}\\
&\leq&C(\varepsilon_{1}+\varepsilon_{2})^{\alpha}\|u-v\|_{L^{\infty}(B_{1})}
+\frac{1}{2}\lambda^{2}C_{0}\|u-v\|_{L^{\infty}(B_{1})}\\
&&+C(\|f(x,u)-f(x,u(0))\|_{L^{\infty}(B_{1})}+T_{1}\varepsilon_{1}+T_{2}\varepsilon_{2})\\
&\leq& C_{1}(\lambda^{2}+(\varepsilon_{1}+
\varepsilon_{2})^{\alpha})\|u-v\|_{L^{\infty}(B_{1})}\\
&&+C_{2}\left(\|f(x,u)-f(x,u(0))\|_{L^{\infty}(B_{1})}+T_{1}\varepsilon_{1}+T_{2}\varepsilon_{2}\right).
\end{eqnarray*}
\end{proof}
\begin{rem}\label{inftyrem}
Under the assumptions of Lemma $\ref{kl}$, we can also show that there exists a second order polynomial $P(x)=E+F\cdot x+x^{T}Gx$ satisfying $a_{ij}(0)D_{ij}P(x)=0$ such that
\begin{eqnarray*}
\|u-v-P(x)\|_{L^{\infty}(B_{\lambda})}&\leq& C_{1}(\lambda^{3}+(\varepsilon_{1}+
\varepsilon_{2})^{\alpha})\|u-v\|_{L^{\infty}(B_{1})}\\
&&+C_{2}\left(\|f(x,u)-f(x,u(0))\|_{L^{\infty}(B_{1})}+T_{1}\varepsilon_{1}+T_{2}\varepsilon_{2}\right),
\end{eqnarray*}
and
$$0<|E|+|F|+|G|\leq C_{0}\|u-v\|_{L^{\infty}(B_{1})}.
$$
To prove this, we only need to substitute first order Taylor polynomial to second order Taylor polynomial of $h$ at 0 in the proof of Lemma $\ref{kl}$.
\end{rem}
\begin{rem}
If all $a_{ij}$ and $b_{i}$ satisfies
$$\|a_{ij}-a_{ij}(0)\|_{L^{\infty}(B_{1})}\leq \varepsilon_{1}, \quad \|b_{i}\|_{L^{\infty}(B_{1})}\leq \varepsilon_{2},
$$
Lemma $\ref{fs}$ and $\ref{kl}$ still hold naturally.
\end{rem}
\begin{rem}
If we consider the strong solution $u\in W_{\text{loc}}^{2,n}(\Omega)\cap C(\overline{\Omega})$  to the equation, for low order coefficients $b_i$, it is sufficient  to assume that $b_{i}\in L^{n}(\Omega)$ with $\|b_{i}\|_{L^{n}(\Omega)}\leq\varepsilon_{2}$ to show the results in Lemma $\ref{fs}$ and $\ref{kl}$, since the H\"{o}lder regularity holds for strong solutions under this kind of assumption by Theorem 2.7 in \cite{S2010}.
\end{rem}

\section{Interior $C^{1}$ regularity}
In this section, we prove the $C^{1}$ regularity of $u$ at 0. The proof is standard. While the treatment of the difficulties resulting from $a_{ij}$ and $b_j$ is subtle. For convenience, we denote $v_{0}$ by $v$. Without loss of generality, we assume that
$$u(0)=0,\quad v(0)=|Dv(0)|=0, \quad r_{0}=1, \quad \int_{0}^{1} \frac{\varphi(r)}{r} d r \leq 1,
$$
$\nu,~\Lambda_{1}$ are small enough with
$$(\nu+\Lambda_{1})^{\alpha}\leq\lambda^{2},
$$
where $\alpha$ is determined in Lemma $\ref{kl}$ and $\lambda$ is small enough and satisfies
\begin{equation}\label{lambda}
0<\lambda<\frac{1}{4},\quad 2C_{1}\lambda<\frac{1}{4},
\end{equation}
$C_{1}$ is the constant in Lemma $\ref{kl}$.\\
\\
{\bf Proof of Theorem $\ref{mr1}$:}
\
We divide this proof into five steps.

\
\

{\bf Step 1:} Approximate $u-v$ by linear functions in different scales.

\
\

We claim that there exist linear functions $\{L_{k}\}_{k=0}^{\infty}$ with $L_{k}=A_{k}+B_{k}\cdot x$ and nonnegative sequences $\{M_{k}\}_{k=0}^{\infty}$, $\{\xi_{k}\}_{k=0}^{\infty}$, $\{\eta_{k}\}_{k=0}^{\infty}$ such that
\begin{equation}\label{induction}
M_{k+1}\leq \xi_{k}M_{k}+\eta_{k},
\end{equation}
and
\begin{eqnarray*}\label{n}
|A_{k+1}-A_{k}|\leq C_{0}\lambda^{k}M_{k},\quad |B_{k+1}-B_{k}|\leq C_{0}M_{k},
\end{eqnarray*}
where for $k=0,1,2,\ldots,$
$$M_{k}=\frac{\|u-v-L_{k}\|_{L^{\infty}(B_{\lambda^{k}})}}{\lambda^{k}},
$$
$$\xi_{k}=\frac{C_{1}}{\lambda}(\lambda^{2}+(\nu+
\Lambda_{1}\lambda^{k(1-\frac{n}{q})})^{ \alpha}),
$$
\begin{eqnarray*}
\eta_{k}=\frac{C_{2}}{\lambda}\left(\lambda^{k}\|f(x,u)-f(x,0)\|_{L^{\infty}(B_{\lambda^{k}})}
+T\lambda^{k}(\nu+\Lambda_{1}\lambda^{k(1-\frac{n}{q})})+\Lambda_{1}\lambda^{k(1-\frac{n}{q})}|B_{k}|\right).
\end{eqnarray*}

Actually, we can prove the claim by induction. For $k=0$, we set $A_0=|B_0|=0$, and
$$
M_0=\|u-v\|_{L^\infty(B_1)},~~~ \xi_{0}=\frac{C_{1}}{\lambda}(\lambda^{2}+(\nu+
\Lambda_{1})^{\alpha}),
$$
and
$$
\eta_{0}=\frac{C_{2}}{\lambda}\left(\|f(x,u)-f(x,0)\|_{L^{\infty}(B_{1})}
+T(\nu+\Lambda_{1})\right).
$$
Since $v(0)=|Dv(0)|=0$, then by {\bf Assumption 2} we yield
$$\|Dv\|_{L^{\infty}(B_{1})}=\|Dv-Dv(0)\|_{L^{\infty}(B_{1})}\leq\|D^{2}v\|_{L^{\infty}(B_{1})}\leq T.
$$
For $a_{ij}$ and $b_{i}$, we notice that $\|a_{ij}-a_{ij}(0)\|_{L^{n}(B_{1})}\leq\nu$, $\|b_{i}\|_{L^{q}(B_{1})}\leq \Lambda_{1}$.
So we can take $T_{1}=T_{2}=T$, $\varepsilon_{1}=\nu$ and $\varepsilon_{2}=\Lambda_{1}$ in Lemma $\ref{kl}$. Hence there exist $A_1$ and $B_1$ with $|A_{1}|+|B_{1}|\leq C_{0}\|u-v\|_{L^{\infty}(B_{1})}$ such that
\begin{eqnarray*}
M_{1}&=&\frac{\|u-v-L_{1}(x)\|_{L^{\infty}(B_{\lambda})}}{\lambda}\\
&\leq&\frac{C_{1}}{\lambda}(\lambda^{2}+(\nu+
\Lambda_{1})^{\alpha})\|u-v\|_{L^{\infty}(B_{1})}\\
&&+\frac{C_{2}}{\lambda}\left(\|f(x,u)-f(x,0)\|_{L^{\infty}(B_{1})}+T(\nu+
\Lambda_{1})\right)\\
&=&\frac{C_{1}}{\lambda}(\lambda^{2}+(\nu+
\Lambda_{1})^{\alpha})\|u-v\|_{L^{\infty}(B_{1})}\\
&&+\frac{C_{2}}{\lambda}\left(\|f(x,u)-f(x,0)\|_{L^{\infty}(B_{1})}+T(\nu+
\Lambda_{1})+\Lambda_{1}|B_{0}|\right)\\
&=&\xi_{0}M_{0}+\eta_{0},
\end{eqnarray*}
and
$$|A_{1}-A_{0}|\leq C_{0}\|u-v\|_{L^{\infty}(B_{1})}=C_{0}M_{0},\quad |B_{1}-B_{0}|\leq C_{0}\|u-v\|_{L^{\infty}(B_{1})}=C_{0}M_{0}.$$
Also $\eta_1$ and $\xi_1$ are well defined from $A_1$ and $B_1$.

Next we assume that the conclusion is true for $k$, that is
$$M_{k}\leq\xi_{k-1}M_{k-1}+\eta_{k-1},$$
and
$$
|A_{k}-A_{k-1}|\leq C_{0}\lambda^{k-1}M_{k-1},\quad |B_{k}-B_{k-1}|\leq C_{0}M_{k-1},
$$
We set $\hat{u}(x)=(u-v-L_{k})(x)$ and consider the equation
\begin{eqnarray*}
a_{ij}D_{ij}\hat{u}+b_{i}D_{i}\hat{u}=h(x)\quad\text{in}~~B_{\lambda^{k}},
\end{eqnarray*}
where $$h(x)=f(x,u)-f(x,0)
+(a_{ij}(0)-a_{ij})D_{ij}v
-b_{i}D_{i}v-b_{i}D_{i}L_{k}.$$
For $z\in B_1$, we set
$$\widetilde{u}(z)=\frac{u(\lambda^{k}z)}{\lambda^{2k}},
\quad\widetilde{v}(z)=\frac{v(\lambda^{k}z)+L_{k}(\lambda^{k}z)}{\lambda^{2k}},
$$
$$\widetilde{f}(z)=f(\lambda^{k} z,u(\lambda^{k} z))-f(\lambda^{k} z,0),
$$
$$\widetilde{a}_{ij}(z)=a_{ij}(\lambda^{k} z),\quad \widetilde{b}_{i}(z)=\lambda^{k} b_{i}(\lambda^{k} z).
$$
Then $\widetilde{u}(z)-\widetilde{v}(z)$ is a solution of
\begin{eqnarray*}
\widetilde{a}_{ij}D_{ij}(\widetilde{u}-\widetilde{v})+\widetilde{b}_{i}D_{i}(\widetilde{u}
-\widetilde{v})=
\widetilde{f}
+(a_{ij}(0)-\widetilde{a}_{ij})D_{ij}\widetilde{v}
-\widetilde{b}_{i}D_{i}\widetilde{v}\qquad&\text{in}~~B_{1}.
\end{eqnarray*}
Combining with the conditions and assumptions on $a_{ij}$, $b_{i}$ and $v$, we also have
\begin{eqnarray*}
& &\|\widetilde{a}_{ij}-a_{ij}(0)\|_{L^{n}(B_{1})}=\frac{1}{\lambda^{k}}\|a_{ij}-a_{ij}(0)\|_{L^{n}(B_{\lambda^{k}})}
\leq\nu,\\
& &\|\widetilde{b}_{i}\|_{L^{q}(B_{1})}=\lambda^{k(1-\frac{n}{q})}\|b_{i}\|_{L^{q}(B_{\lambda^{k}})}
\leq\Lambda_{1}\lambda^{k(1-\frac{n}{q})},\\
& & \|D^{2}\widetilde{v}\|_{L^{\infty}(B_{1})}=\|D^{2}v\|_{L^{\infty}(B_{\lambda^{k}})}\leq T,\\
& & \|D\widetilde{v}\|_{L^{\infty}(B_{1})}=\frac{\|D(v+L_{k})\|_{L^{\infty}(B_{\lambda^{k}})}}{\lambda^{k}}
\leq\frac{T\lambda^{k}+|B_{k}|}{\lambda^{k}}.
\end{eqnarray*}
Then by Lemma $\ref{kl}$, there exists a linear function $L(z)=A+B\cdot z$ such that
\begin{eqnarray*}
\|\widetilde{u}-\widetilde{v}-L(z)\|_{L^{\infty}(B_{\lambda})}
&\leq& C_{1}(\lambda^{2}+(\nu+
\Lambda_{1}\lambda^{k(1-\frac{n}{q})})^{\alpha})\|\widetilde{u}-\widetilde{v}\|_{L^{\infty}(B_{1})}\\
&&+C_{2}\left(\|\widetilde{f}\|_{L^{\infty}(B_{1})}+T\nu+(T+\frac{|B_{k}|}{\lambda^{k}})\Lambda_{1}\lambda^{k(1-\frac{n}{q})}
\right),
\end{eqnarray*}
and
$$|A|+|B|\leq C_{0}\|\widetilde{u}-\widetilde{v}\|_{L^{\infty}(B_{1})}=C_{0}\frac{\|u-v-L_{k}\|_{L^{\infty}(B_{\lambda^{k}})}}{\lambda^{2k}}=C_{0}\frac{M_{k}}{\lambda^{k}}.
$$
Let $L_{k+1}(x)=L_{k}(x)+\lambda^{2k}L(\displaystyle\frac{x}{\lambda^{k}})$. If we take the scale back, then we get
\begin{eqnarray*}
M_{k+1}&=&\frac{\|u-v-L_{k+1}\|_{L^{\infty}(B_{\lambda^{k+1}})}}{\lambda^{k+1}}\\
&\leq& \frac{C_{1}(\lambda^{2}+(\nu+
\Lambda_{1}\lambda^{k(1-\frac{n}{q})})^{\alpha})}{\lambda}\frac{\|u-v-L_{k}\|_{L^{\infty}(B_{\lambda^{k}})}}{\lambda^{k}}\\
&&+\frac{C_{2}}{\lambda}\left(\lambda^{k}\|f(x,u)-f(x,0)\|_{L^{\infty}(B_{\lambda^{k}})}
+T\lambda^{k}(\nu+\Lambda_{1}\lambda^{k(1-\frac{n}{q})})+\Lambda_{1}\lambda^{k(1-\frac{n}{q})}|B_{k}|\right)\\
&=&\xi_{k}M_{k}+\eta_{k},
\end{eqnarray*}
and
$$|A_{k+1}-A_{k}|=\lambda^{2k}|A|\leq C_{0}\lambda^{k}M_{k},$$
$$|B_{k+1}-B_{k}|=\lambda^{2k}\left|\frac{B}{\lambda^{k}}\right|\leq C_{0}M_{k}.
$$
This implies the conclusion is true for $k+1$. Thus we finish to prove the claim.

\
\

{\bf Step 2:} Prove that $\sum\limits_{k=0}^{\infty}M_{k}<\infty$.

\
\

For $k\geq0,$ we suppose $S_{k}=\sum\limits_{i=0}^{k}M_{i}.$ By {\bf Step 1}, noting that $A_{0}=|B_{0}|=0$, $M_{0}=\|u-v\|_{L^{\infty}(B_{1})}$, then for any $k\geq0,$ we have
\begin{equation}\label{nn}
|B_{k+1}|\leq |B_{k}|+C_{0}M_{k}\leq C_{0}S_{k},
\end{equation}
\begin{equation}\label{eak}
 |A_{k+1}|\leq |A_{k}|+C_{0}\lambda^{k}M_{k}\leq C_{0}S_{k},
\end{equation}
$$\xi_{k}\leq2C_{1}\lambda\leq \frac{1}{4},
$$
where $(\nu+\Lambda_{1})^{\alpha}\leq\lambda^{2}
$ is used. Then the iteration formula $(\ref{induction})$ implies that
$$M_{k+1}\leq\xi_{k}M_{k}+\eta_{k}\leq\frac{1}{4}M_{k}+\eta_{k}.
$$
Now we estimate $\eta_{k}$. By {\bf Assumption \ref{as1}}, we have
\begin{equation}\label{m2}
\eta_{k}\leq\frac{C_{2}}{\lambda}\left(\lambda^{k}\varphi(\|u\|_{L^{\infty}(B_{\lambda^{k}})})
+T\lambda^{k}(\nu+\Lambda_{1}\lambda^{k(1-\frac{n}{q})})+\Lambda_{1}\lambda^{k(1-\frac{n}{q})}|B_{k}|\right).
\end{equation}
Recalling the property of the modulus of continuity (see $(\ref{diniproperty})$) we have
\begin{eqnarray*}
\varphi(\|u\|_{L^{\infty}(B_{\lambda^{k}})})&\leq&\varphi\left(\|u-v-L_{k}\|_{L^{\infty}(B_{\lambda^{k}})}+
\|v\|_{L^{\infty}(B_{\lambda^{k}})}+\|L_{k}\|_{L^{\infty}(B_{\lambda^{k}})}\right)\\
&\leq&\varphi(M_{k}\lambda^{k}+T\lambda^{2k}+|A_{k}|+|B_{k}|\lambda^{k})\\
&\leq&2\left(M_{k}+T\lambda^{k}+\frac{|A_{k}|}{\lambda^{k}}+|B_{k}|\right)\varphi(\lambda^{k}).
\end{eqnarray*}
By substituting the above inequality and $(\ref{nn})$, $(\ref{eak})$ into $(\ref{m2})$, we obtain for $k\geq1$,
\begin{eqnarray*}
\eta_{k}&\leq&\frac{C_{2}}{\lambda}\left(2\lambda^{k}\left(M_{k}+T\lambda^{k}+\frac{|A_{k}|}{\lambda^{k}}
+|B_{k}|\right)\varphi(\lambda^{k})
+T\lambda^{k}(\nu+\Lambda_{1}\lambda^{k(1-\frac{n}{q})})+\Lambda_{1}\lambda^{k(1-\frac{n}{q})}|B_{k}|\right)\\
&\leq&\frac{C_{2}}{\lambda}\left(2(\lambda^{k}M_{k}+|A_{k}|+\lambda^{k}|B_{k}|)\varphi(\lambda^{k})+2T\lambda^{2k}\varphi(\lambda^{k})
+T\lambda^{k}(\nu+\Lambda_{1}\lambda^{k(1-\frac{n}{q})})
+\Lambda_{1}\lambda^{k(1-\frac{n}{q})}|B_{k}|\right)\\
&\leq&\frac{2C_{2}}{\lambda}(\varphi(\lambda^{k})+\Lambda_{1}\lambda^{k(1-\frac{n}{q})})(M_{k}+|A_{k}|+|B_{k}|)+\frac{2C_{2}T}{\lambda}
(\varphi(\lambda^{k})+\Lambda_{1}\lambda^{k(1-\frac{n}{q})}+\nu\lambda^{k})\\
&\leq&\frac{4C_{2}(C_{0}+1)}{\lambda}(\varphi(\lambda^{k})+\Lambda_{1}\lambda^{k(1-\frac{n}{q})})S_{k}
+\frac{2C_{2}T}{\lambda}(\varphi(\lambda^{k})+\Lambda_{1}\lambda^{k(1-\frac{n}{q})}+\nu\lambda^{k}).
\end{eqnarray*}

Then we take $k_{0}$ large enough (then fixed) such that
$$
\sum_{i=k_{0}}^{\infty}\frac{4C_{2}(C_{0}+1)}{\lambda}(\varphi(\lambda^{i})+\Lambda_{1}\lambda^{k(1-\frac{n}{q})})\leq \frac{4C_{2}(C_{0}+1)}{\lambda}\left( \frac{1}{\ln \frac{1}{\lambda}}\int_{0}^{\lambda^{k_{0}-1}} \frac{\varphi(r)}{r} d r+
\frac{\Lambda_{1}\lambda^{k_{0}(1-\frac{n}{q})}}{1-\lambda^{1-\frac{n}{q}}} \right)\leq \frac{1}{4}.
$$
For such $k_{0}(\geq 1),$ we have
$$\sum_{i=k_{0}}^{\infty}\varphi(\lambda^{i})
\leq \frac{1}{\ln \frac{1}{\lambda}} \int_{0}^{1} \frac{\varphi(r)}{r} d r\leq \frac{1}{\ln \frac{1}{\lambda}}.
$$
Therefore for each $k \geq k_{0},$ we have
\begin{eqnarray*}
\sum_{i=k_{0}}^{k}\eta_{i}&\leq&\sum_{i=k_{0}}^{k}\frac{4C_{2}(C_{0}+1)}{\lambda}(\varphi(\lambda^{i})
+\Lambda_{1}\lambda^{i(1-\frac{n}{q})})S_{i}
+\sum_{i=k_{0}}^{k}\frac{2C_{2}T}{\lambda}(\varphi(\lambda^{i})+\Lambda_{1}\lambda^{i(1-\frac{n}{q})}+\nu\lambda^{i})\\
&\leq&\frac{1}{4}S_{k+1}+\frac{2C_{2}T}{\lambda\ln \frac{1}{\lambda}}+\frac{2C_{2}T\nu\lambda^{k_{0}}}{\lambda(1-\lambda)}
+\frac{2C_{2}T\Lambda_{1}\lambda^{k_{0}(1-\frac{n}{q})}}{\lambda(1-\lambda^{1-\frac{n}{q}})}.
\end{eqnarray*}
It follows that
\begin{eqnarray*}
S_{k+1}-S_{k_{0}}=\sum_{i=k_{0}}^{k} M_{i+1}
&\leq&\frac{1}{4}\sum_{i=k_{0}}^{k} M_{i}
+\sum_{i=k_{0}}^{k}\eta_{i}\\
&\leq&\frac{1}{4}S_{k+1}+\frac{1}{4}S_{k+1}+\frac{2TC_{2}}{\lambda\ln \frac{1}{\lambda}}+\frac{2C_{2}T\nu\lambda^{k_{0}}}{\lambda(1-\lambda)}+\frac{2C_{2}T\Lambda_{1}\lambda^{k_{0}(1-\frac{n}{q})}}{\lambda(1-\lambda^{1-\frac{n}{q}})}\\
&=&\frac{1}{2}S_{k+1}+\frac{2TC_{2}}{\lambda\ln \frac{1}{\lambda}}+\frac{2C_{2}T\nu\lambda^{k_{0}}}{\lambda(1-\lambda)}
+\frac{2C_{2}T\Lambda_{1}\lambda^{k_{0}(1-\frac{n}{q})}}{\lambda(1-\lambda^{1-\frac{n}{q}})}.
\end{eqnarray*}
Then for all $k \geq k_{0}$,
$$
S_{k+1} \leq \frac{4TC_{2}}{\lambda\ln \frac{1}{\lambda}}+\frac{4C_{2}T\nu\lambda^{k_{0}}}{\lambda(1-\lambda)}+
\frac{4C_{2}T\Lambda_{1}\lambda^{k_{0}(1-\frac{n}{q})}}{\lambda(1-\lambda^{1-\frac{n}{q}})}+2S_{k_{0}}.
$$
Therefore $\left\{S_{k}\right\}_{k=0}^{\infty}$ is bounded, and consequently   $\left\{S_{k}\right\}_{k=0}^{\infty}$ and $\sum\limits_{k=0}^{\infty} M_{k}$ are convergent.  Moreover, from  $(\ref{nn})$ and $(\ref{eak})$, we know that  $\left\{B_{k}\right\}_{k=0}^{\infty}$ and  $\left\{A_{k}\right\}_{k=0}^{\infty}$ are bounded. It is also clear that $\lim\limits_{k\rightarrow\infty}M_{k}=0$.

\
\

{\bf Step 3:} Prove that $\lim\limits_{k \rightarrow+\infty} A_{k} $ and $\lim\limits_{k \rightarrow+\infty} B_{k} $ exist.

\
\

Actually, it is clear that
$$
\sum_{k=0}^{\infty}|B_{k+1}-B_{k}| \leq C_{0}\sum_{k=0}^{\infty}M_{k}<\infty.
$$
Therefore $\left\{B_{k}\right\}_{k=0}^{\infty}$ is a Cauchy sequence, and consequently  $\lim\limits_{k \rightarrow+\infty} B_{k}$ exists. We set $\overline{B}:=$
$\lim\limits_{k \rightarrow+\infty} B_{k} $. Similarly, we can set  $\overline{A}:=$
$\lim\limits_{k \rightarrow+\infty} A_{k} $. Then we denote $\overline{L}(x)=\overline{A}+\overline{B}\cdot x$.

\
\

{\bf Step 4:} Prove that $u-v$ is $C^{1}$ at 0.

\
\

It is sufficient to show  that  there exists a sequence $\{N_{k}\}$ satisfying   that $\lim\limits_{k \rightarrow \infty} N_{k}=0$ and \begin{equation}\label{hh}\frac{\left\|u-v-\overline{L}\right\|_{L^{\infty}(B_{\lambda^{k}})}}{\lambda^{k}} \leq N_{k}.
\end{equation}

In fact, to show (\ref{hh}), we notice that for any $k \geq 0$,
$$
\frac{\left\|u-v-\overline{L}\right\|_{L^{\infty}(B_{\lambda^{k}})}}{\lambda^{k}}
\leq\frac{\|u-v-L_{k}\|_{L^{\infty}(B_{\lambda^{k}})}}{\lambda^{k}}
+\frac{\|L_{k}-\overline{L}\|_{L^{\infty}(B_{\lambda^{k}})}}{\lambda^{k}}.
$$
By the definition of $M_{k}$ we get
\begin{eqnarray*}
\frac{\left\|u-v-\overline{L}\right\|_{L^{\infty}(B_{\lambda^{k}})}}{\lambda^{k}}
&\leq& M_{k}+\frac{|A_{k}-\overline{A}|+|B_{k}-\overline{B}|\lambda^{k}}{\lambda^{k}}\\
&\leq& M_{k}+|B_{k}-\overline{B}|+\frac{|A_{k}-\overline{A}|}{\lambda^{k}}.
\end{eqnarray*}
We set $$N_{k}=M_{k}+|B_{k}-\overline{B}|+\displaystyle\frac{|A_{k}-\overline{A}|}{\lambda^{k}}.$$
Then by using $(\ref{eak})$, we have
\begin{eqnarray*}
|A_{k}-\overline{A}|&\leq&|A_{k}-A_{k+1}|+|A_{k+1}-A_{k+2}|+\cdots\\
&\leq& C_{0}\left(\lambda^{k}M_{k}+\lambda^{k+1}M_{k+1}+\cdots\right)\\
&\leq& C_{0}\lambda^{k}\sum\limits_{i=k}^{\infty}M_{i}.
\end{eqnarray*}
It follows that
$$\frac{|A_{k}-\overline{A}|}{\lambda^{k}}\leq C_{0}\sum\limits_{i=k}^{\infty}M_{i}.
$$
Since $\sum\limits_{i=0}^{\infty}M_{i}$ is convergent, then we have $\displaystyle\lim\limits_{k\rightarrow\infty}\frac{|A_{k}-\overline{A}|}{\lambda^{k}}=0$. Consequently we obtain that  $\lim\limits_{k\rightarrow\infty}N_{k}=0$.

\
\

{\bf Step 5:} Prove that $u$ is $C^{1}$ at 0.

\
\

From {\bf Step 4}, we know that $u-v$ is $C^{1}$ at 0. Since $v$ is $C^{1,1}$ at 0, then $u$ is $C^{1}$ at 0.

\section{Interior $C^{1,1}$ regularity}

In this section, we prove $C^{1,1}$ regularity of $u$ at 0. Similar as Section 3, we firstly give some normalization. We assume
$u(0)=0$, $v(0)=|Dv(0)|=0$, $r_{0}=1$, and
$$ \quad \int_{0}^{1} \frac{\varphi(r)}{r} d r \leq 1, \quad \int_{0}^{1} \frac{\omega_{1}(r)}{r} d r \leq 1, \quad \int_{0}^{1} \frac{\omega_{2}(r)}{r} d r \leq 1.
$$
Furthermore, for $b_{i}\in L^{\infty}(B_{1})$, we can assume that there exists constant $\tau$ small enough such that
\begin{equation}\label{bsmall}
\|b_{i}\|_{L^{\infty}(B_{1})}\leq\tau,\quad \frac{\tau C_{0}}{2\Lambda}\leq\frac{1}{4},\quad (\omega_{1}(1)+
\tau)^{\alpha}\leq\lambda^{3}.
\end{equation}
where $\alpha$ and $C_{0}$ are determined by Lemma $\ref{kl}$ and Remark $\ref{inftyrem}$, $\lambda$ is small enough and satisfies
\begin{equation}\label{lambda1}
0<\lambda<\frac{1}{4},\quad 2C_{1}\lambda<\frac{1}{4},
\end{equation}
$C_{1}$ is the constant in Remark $\ref{inftyrem}$.\\
\\
{\bf Proof of Theorem $\ref{mr2}$:} We divide the proof into five steps.

\
\

{\bf Step 1:} Approximate $u-v$ by second order polynomials in different scales.

\
\

We claim that there exist second order polynomials $\{P_{k}\}_{k=0}^{\infty}$ with $P_{k}=E_{k}+F_{k}\cdot x+x^{T}G_{k}x$ and nonnegative sequences $\{M_{k}\}_{k=0}^{\infty}$, $\{\xi_{k}\}_{k=0}^{\infty}$, $\{\eta_{k}\}_{k=0}^{\infty}$
such that
\begin{equation}\label{induction1}
M_{k+1}\leq \xi_{k}M_{k}+\eta_{k},
\end{equation}
and
\begin{eqnarray*}\label{n}
|E_{k+1}-E_{k}|\leq C_{0}\lambda^{2k}M_{k},\quad |F_{k+1}-F_{k}|\leq C_{0}\lambda^{k}M_{k},\quad |G_{k+1}-G_{k}|\leq C_{0}M_{k},
\end{eqnarray*}
where for $k=0,1,2,\ldots,$
$$M_{k}=\frac{\|u-v-P_{k}+\vec{b}(0)\cdot F_{k}\frac{x_{1}^{2}}{2a_{11}(0)}\|_{L^{\infty}(B_{\lambda^{k}})}}{\lambda^{2k}},
$$
$$\xi_{k}=\frac{C_{1}}{\lambda^{2}}(\lambda^{3}+(\omega_{1}(\lambda^{k})+
\tau\lambda^{k})^{\alpha}),
$$
\begin{eqnarray*}
\eta_{k}&=&\frac{C_{2}}{\lambda^{2}}\left(\|f(x,u)-f(x,0)\|_{L^{\infty}(B_{\lambda^{k}})}
+(\omega_{1}(\lambda^{k})+\tau\lambda^{k})\left(T+2|G_{k}|+\frac{\tau}{\Lambda}|F_{k}|\right)+\omega_{2}(\lambda^{k})|F_{k}|\right)\\
&&+\frac{\tau}{2\Lambda}|F_{k+1}-F_{k}|.
\end{eqnarray*}

\
\

Actually, we can prove the claim by induction.

When $k=0$, we choose $P_{0}=0$, i.e. $E_{0}=|F_{0}|=|G_{0}|=0$. If we take $T_{1}=T_{2}=T$, $\varepsilon_{1}=\omega_{1}(1)$, $\varepsilon_{2}=\tau$ in Remark $\ref{inftyrem}$, then there exists $P_{1}(x)=E_{1}+F_{1}\cdot x+x^{T}G_{1}x$ satisfying $a_{ij}(0)D_{ij}P_{1}(x)=0$ with $|E_{1}|+|F_{1}|+|G_{1}|\leq C_{0}\|u-v\|_{L^{\infty}(B_{1})}$ such that
\begin{eqnarray*}
\frac{\|u-v-P_{1}(x)\|_{L^{\infty}(B_{\lambda})}}{\lambda^{2}}
&\leq&\frac{C_{1}}{\lambda^{2}}(\lambda^{3}+(\omega_{1}(1)+
\tau)^{\alpha})\|u-v\|_{L^{\infty}(B_{1})}\\
&&+\frac{C_{2}}{\lambda^{2}}\left(\|f(x,u)-f(x,0)\|_{L^{\infty}(B_{1})}+T(\omega_{1}(1)+
\tau)\right).
\end{eqnarray*}
It follows that
\begin{eqnarray*}
M_{1}&=&\frac{\|u-v-P_{1}(x)+\vec{b}(0)\cdot F_{1}\frac{x_{1}^{2}}{2a_{11}(0)}\|_{L^{\infty}(B_{\lambda})}}{\lambda^{2}}\\
&\leq&\frac{\|u-v-P_{1}(x)\|_{L^{\infty}(B_{\lambda})}}{\lambda^{2}}
+\frac{\|\vec{b}(0)\cdot F_{1}\frac{x_{1}^{2}}{2a_{11}(0)}\|_{L^{\infty}(B_{\lambda})}}{\lambda^{2}}\\
&\leq&\frac{C_{1}}{\lambda^{2}}(\lambda^{3}+(\omega_{1}(1)+
\tau)^{\alpha})\|u-v-P_{0}+\vec{b}(0)\cdot F_{0}\frac{x_{1}^{2}}{2a_{11}(0)}\|_{L^{\infty}(B_{1})}\\
&&+\frac{C_{2}}{\lambda^{2}}\left(\|f(x,u)-f(x,0)\|_{L^{\infty}(B_{1})}+\left(T+2|G_{0}|+\frac{\tau}{\Lambda}|F_{0}|\right)
(\omega_{1}(1)+
\tau)+\omega_{2}(1)|F_{0}|\right)\\
&&+\frac{\tau}{2\Lambda}|F_{1}-F_{0}|\\
&=&\xi_{0}M_{0}+\eta_{0},
\end{eqnarray*}
and
$$|E_{1}-E_{0}|\leq C_{0}\|u-v-P_{0}-\vec{b}(0)\cdot F_{0}\frac{x_{1}^{2}}{2a_{11}(0)}\|_{L^{\infty}(B_{1})}=C_{0}M_{0},
$$
similarly
$$|F_{1}-F_{0}|\leq C_{0}M_{0},\quad|G_{1}-G_{0}|\leq C_{0}M_{0}.
$$

Next we assume that the conclusion is true for $k$, that is
$$M_{k}\leq\xi_{k-1}M_{k-1}+\eta_{k-1},$$
and
\begin{eqnarray*}\label{n}
|E_{k}-E_{k-1}|\leq C_{0}\lambda^{2(k-1)}M_{k-1},\quad |F_{k}-F_{k-1}|\leq C_{0}\lambda^{k-1}M_{k-1},\quad |G_{k}-G_{k-1}|\leq C_{0}M_{k-1}.
 \end{eqnarray*}
Denote $\hat{u}=u-P_{k}(x)+\vec{b}(0)\cdot F_{k}\frac{x_{1}^{2}}{2a_{11}(0)}$. Then $\hat{u}$ satisfies that
\begin{eqnarray*}
a_{ij}D_{ij}\hat{u}+b_{i}D_{i}\hat{u}=h(x)\quad\text{in}~~B_{\lambda^{k}},
\end{eqnarray*}
where
$$h(x)=f(x,u)
+(a_{ij}(0)-a_{ij})D_{ij}(P_{k}-\vec{b}(0)\cdot F_{k}\frac{x_{1}^{2}}{2a_{11}(0)})+\vec{b}(0)\cdot F_{k}
-b_{i}D_{i}(P_{k}(x)-\vec{b}(0)\cdot F_{k}\frac{x_{1}^{2}}{2a_{11}(0)}).
$$
For $z\in B_1$, we set
$$\widetilde{u}(z)=\frac{\hat{u}(\lambda^{k}z)}{\lambda^{2k}}
\quad\widetilde{v}(z)=\frac{v(\lambda^{k}z)}{\lambda^{2k}},
$$
$$\widetilde{f}(z)=h(\lambda^{k}z)-f(\lambda^{k} z,0),
$$
$$\widetilde{a}_{ij}(z)=a_{ij}(\lambda^{k} z),\quad \widetilde{b}_{i}(z)=\lambda^{k} b_{i}(\lambda^{k} z).
$$
Then $\widetilde{u}(z)-\widetilde{v}(z)$ is a solution of
\begin{eqnarray*}
\widetilde{a}_{ij}D_{ij}(\widetilde{u}-\widetilde{v})+\widetilde{b}_{i}D_{i}(\widetilde{u}
-\widetilde{v})=
\widetilde{f}
+(a_{ij}(0)-\widetilde{a}_{ij})D_{ij}\widetilde{v}
-\widetilde{b}_{i}D_{i}\widetilde{v}\quad\text{in}~~B_{1}.
\end{eqnarray*}
Combining with the conditions on $a_{ij}$, $b_{i}$ and {\bf Assumption 2} on $v$, we also have
$$\|a_{ij}(0)-\widetilde{a}_{ij}\|_{L^{\infty}(B_{1})}=
\|a_{ij}(0)-a_{ij}\|_{L^{\infty}(B_{\lambda^{k}})}\leq\omega_{1}(\lambda^{k}),
$$
$$\|\widetilde{b}_{i}\|_{L^{\infty}(B_{1})}=\lambda^{k}\|b_{i}\|_{L^{\infty}(B_{\lambda^{k}})}\leq\tau\lambda^{k},
$$
$$\widetilde{v}(0)=|D\widetilde{v}(0)|=0,\quad\|D^{2}\widetilde{v}\|_{L^{\infty}(B_{1})}=\|D^{2}v\|_{L^{\infty}(B_{\lambda^{k}})}
\leq T,
$$
$$
\|D\widetilde{v}\|_{L^{\infty}(B_{1})}=\frac{\|Dv\|_{L^{\infty}(B_{\lambda^{k}})}}{\lambda^{k}}\leq\frac{T\lambda^{k}}
{\lambda^{k}}\leq T.
$$
Then by Lemma $\ref{kl}$ and Remark $\ref{inftyrem}$, there exists a second order polynomial $P(z)=E+F\cdot z+z^{T}Gz$ satisfying $\widetilde{a}_{ij}(0)D_{ij}P(z)=0$ such that
\begin{eqnarray*}
\|\widetilde{u}-\widetilde{v}-P(z)\|_{L^{\infty}(B_{\lambda})}
&\leq& C_{1}(\lambda^{3}+(\omega_{1}(\lambda^{k})+
\tau\lambda^{k})^{\alpha})\|\widetilde{u}-\widetilde{v}\|_{L^{\infty}(B_{1})}\\
&&+C_{2}\left(\|\widetilde{f}\|_{L^{\infty}(B_{1})}+T(\omega_{1}(\lambda^{k})+\tau\lambda^{k})
\right),
\end{eqnarray*}
and
$$|E|+|F|+|G|\leq C_{0}\|\widetilde{u}-\widetilde{v}\|_{L^{\infty}(B_{1})}=C_{0}\frac{\|u-v-P_{k}+\vec{b}(0)\cdot F_{k}\frac{x_{1}^{2}}
{2a_{11}(0)}\|_{L^{\infty}(B_{\lambda^{k}})}}{\lambda^{2k}}=C_{0}M_{k},
$$
where
\begin{eqnarray*}
\|\widetilde{f}\|_{L^{\infty}(B_{1})}&=&\|h(\lambda^{k}z)-f(\lambda^{k} z,0)\|_{L^{\infty}(B_{1})}\\
&\leq&\|f(x,u)-f(x,0)\|_{L^{\infty}(B_{\lambda^{k}})}+\|2(a_{ij}(0)-a_{ij})G_{k}\|_{L^{\infty}(B_{\lambda^{k}})}\\
&&+\frac{\|(a_{11}(0)-a_{11})\vec{b}(0)\cdot F_{k}\|_{L^{\infty}(B_{\lambda^{k}})}}{a_{11}(0)}
+\|(\vec{b}(0)-\vec{b})\cdot F_{k}\|_{L^{\infty}(B_{\lambda^{k}})}\\
&&+\|2\vec{b}\cdot  (G_{k}x)\|_{L^{\infty}(B_{\lambda^{k}})}
+\frac{\|b_{1}\vec{b}(0)\cdot F_{k}x_{1}\|_{L^{\infty}(B_{\lambda^{k}})}}{a_{11}(0)}\\
&\leq&\|f(x,u)-f(x,0)\|_{L^{\infty}(B_{\lambda^{k}})}+2\omega_{1}(\lambda^{k})|G_{k}|\\
&&+\frac{\tau}{\Lambda}\omega_{1}(\lambda^{k})|F_{k}|+\omega_{2}(\lambda^{k})|F_{k}|+2\tau\lambda^{k}|G_{k}|+
\frac{\tau^{2}}{\Lambda}\lambda^{k}|F_{k}|\\
&\leq&\|f(x,u)-f(x,0)\|_{L^{\infty}(B_{\lambda^{k}})}+(\omega_{1}(\lambda^{k})
+\tau\lambda^{k})(2|G_{k}|+\frac{\tau}{\Lambda}|F_{k}|)+\omega_{2}(\lambda^{k})|F_{k}|.
\end{eqnarray*}
Let $P_{k+1}(x)=P_{k}(x)+\lambda^{2k}P(\displaystyle\frac{x}{\lambda^{k}})$.
If we take the scale back, then we get
\begin{eqnarray*}
&&\frac{\|u-v-P_{k+1}+\vec{b}(0)\cdot F_{k}\frac{x_{1}^{2}}{2a_{11}(0)}\|_{L^{\infty}(B_{\lambda^{k+1}})}}{\lambda^{2(k+1)}}\\
&\leq& \frac{C_{1}(\lambda^{3}+(\omega_{1}(\lambda^{k})+
\tau\lambda^{k})^{\alpha})}{\lambda^{2}}\frac{\|u-v-P_{k}+\vec{b}(0)\cdot F_{k}\frac{x_{1}^{2}}
{2a_{11}(0)}\|_{L^{\infty}(B_{\lambda^{k}})}}{\lambda^{2k}}\\
&&+\frac{C_{2}}{\lambda^{2}}\|f(x,u)-f(x,0)\|_{L^{\infty}(B_{\lambda^{k}})}\\
&&+\frac{C_{2}}{\lambda^{2}}((\omega_{1}(\lambda^{k})
+\tau\lambda^{k})(T+2|G_{k}|+\frac{\tau}{\Lambda}|F_{k}|)+\omega_{2}(\lambda^{k})|F_{k}|).
\end{eqnarray*}
Then we have
\begin{eqnarray*}
M_{k+1}&=&\frac{\|u-v-P_{k+1}+\vec{b}(0)\cdot F_{k+1}\frac{x_{1}^{2}}{2a_{11}(0)}\|_{L^{\infty}(B_{\lambda^{k+1}})}}{\lambda^{2(k+1)}}\\
&\leq&\frac{\|u-v-P_{k+1}+\vec{b}(0)\cdot F_{k}\frac{x_{1}^{2}}{2a_{11}(0)}\|_{L^{\infty}(B_{\lambda^{k+1}})}}{\lambda^{2(k+1)}}
+\frac{\|\vec{b}(0)\cdot F_{k}\frac{x_{1}^{2}}{2a_{11}(0)}-\vec{b}(0)\cdot F_{k+1}\frac{x_{1}^{2}}{2a_{11}(0)}\|_{L^{\infty}(B_{\lambda^{k+1}})}}
{\lambda^{2(k+1)}}\\
&\leq&\frac{C_{1}(\lambda^{3}+(\omega_{1}(\lambda^{k})+
\tau\lambda^{k})^{\alpha})}{\lambda^{2}}\frac{\|u-v-P_{k}+\vec{b}(0)\cdot F_{k}\frac{x_{1}^{2}}
{2a_{11}(0)}\|_{L^{\infty}(B_{\lambda^{k}})}}{\lambda^{2k}}\\
&&+\frac{C_{2}}{\lambda^{2}}\|f(x,u)-f(x,0)\|_{L^{\infty}(B_{\lambda^{k}})}\\
&&+\frac{C_{2}}{\lambda^{2}}((\omega_{1}(\lambda^{k})
+\tau\lambda^{k})(T+2|G_{k}|+\frac{\tau}{\Lambda}|F_{k}|)+\omega_{2}(\lambda^{k})|F_{k}|)
+\frac{\tau}{2\Lambda}|F_{k+1}-F_{k}|\\
&=&\xi_{k}M_{k}+\eta_{k},
\end{eqnarray*}
and
$$|E_{k+1}-E_{k}|=\lambda^{2k}|E|\leq C_{0}\lambda^{2k}M_{k},$$
$$|F_{k+1}-F_{k}|=\lambda^{2k}\left|\frac{F}{\lambda^{k}}\right|\leq C_{0}\lambda^{k}M_{k},
$$
$$|G_{k+1}-G_{k}|=\lambda^{2k}\left|\frac{G}{\lambda^{2k}}\right|\leq C_{0}M_{k}.
$$
This implies the conclusion is true for $k+1$. Thus we finish to prove the claim.

\
\

{\bf Step 2:} Prove that $\sum\limits_{k=0}^{\infty}M_{k}<\infty$.

\
\

For $k\geq0,$ we suppose $S_{k}=\sum\limits_{i=0}^{k}M_{i}.$ By {\bf Step 1}, noting that $E_{0}=|F_{0}|=|G_{0}|=0$, $M_{0}=\|u-v\|_{L^{\infty}(B_{1})}$, then for any $k\geq0,$ we have
\begin{equation}\label{gk}
|G_{k+1}|\leq |G_{k}|+C_{0}M_{k}\leq C_{0}S_{k},
\end{equation}
\begin{equation}\label{fk}
|F_{k+1}|\leq |F_{k}|+C_{0}\lambda^{k}M_{k}\leq C_{0}S_{k},
\end{equation}
\begin{equation}\label{ek}
|E_{k+1}|\leq |E_{k}|+C_{0}\lambda^{2k}M_{k}\leq C_{0}S_{k},
\end{equation}
$$\xi_{k}\leq2C_{1}\lambda\leq \frac{1}{4},
$$
where $(\omega_{1}(1)+
\tau)^{\alpha}\leq\lambda^{3}
$ is used. Then the iteration formula $(\ref{induction1})$ implies that
$$M_{k+1}\leq\xi_{k}M_{k}+\eta_{k}\leq\frac{1}{4}M_{k}+\eta_{k}.
$$
Now we estimate $\eta_{k}$. By {\bf Assumption \ref{as1}}, we have
\begin{equation}\label{m2-1}
\begin{aligned}
\eta_{k}&\leq\frac{C_{2}}{\lambda^{2}}\left(\varphi(\|u\|_{L^{\infty}(B_{\lambda^{k}})})
+(\omega_{1}(\lambda^{k})+\tau\lambda^{k})\left(T+2|G_{k}|+\frac{\tau}{\Lambda}|F_{k}|\right)
+\omega_{2}(\lambda^{k})|F_{k}|\right)\\
&\quad+\frac{\tau}{2\Lambda}|F_{k+1}-F_{k}|.
\end{aligned}
\end{equation}
Since $\|b_{i}\|_{L^{\infty}(B_{1})}\leq\tau$, obviously $\|b_{i}\|_{L^{q}(B_{1})}\leq\tau |B_{1}|^{\frac{1}{q}}$. For $a_{ij}$, we can calculate that
$$\|a_{ij}-a_{ij}(0)\|_{L^{n}(B_{r})}\leq\omega_{1}(1)|B_{r}|^{\frac{1}{n}},\quad\forall~0<r\leq1.
$$
Thus Theorem $\ref{mr1}$ implies that $u$ is $C^{1}$ at $0$, and
$$\|u(x)\|_{L^{\infty}(B_{\lambda^{k}})}=\|u(x)-u(0)\|_{L^{\infty}(B_{\lambda^{k}})}\leq C^{*}\lambda^{k}.
$$
where $C^{*}$ is a constant depending only on $n,\Lambda,\omega_{1}(1),\tau$ and $T$. Recalling $(\ref{diniproperty})$ we have
\begin{eqnarray*}
\varphi(\|u\|_{L^{\infty}(B_{\lambda^{k}})})\leq\varphi\left(C^{*}\lambda^{k}\right)\leq2C^{*}\varphi(\lambda^{k}).
\end{eqnarray*}
By substituting the above inequality and $(\ref{gk})$, $(\ref{fk})$, $(\ref{ek})$ into $(\ref{m2-1})$, we obtain for $k\geq1$,
\begin{eqnarray*}
\eta_{k}&\leq&\frac{C_{2}}{\lambda^{2}}\left(2C^{*}\varphi(\lambda^{k})
+(\omega_{1}(\lambda^{k})+\tau\lambda^{k})\left(T+2|G_{k}|+\frac{\tau}{\Lambda}|F_{k}|\right)
+\omega_{2}(\lambda^{k})|F_{k}|\right)+\frac{\tau}{2\Lambda}|F_{k+1}-F_{k}|\\
&\leq&\frac{2C^{*}C_{2}}{\lambda^{2}}\varphi(\lambda^{k})+\frac{C_{2}}{\lambda^{2}}\left(
(\omega_{1}(\lambda^{k})+\tau\lambda^{k})\left(T+2C_{0}S_{k}+\frac{\tau}{\Lambda}C_{0}S_{k}\right)
+\omega_{2}(\lambda^{k})C_{0}S_{k}\right)+\frac{\tau}{2\Lambda}C_{0}M_{k}\\
&\leq&\frac{2C^{*}C_{2}}{\lambda^{2}}\varphi(\lambda^{k})+\frac{C_{2}T}{\lambda^{2}}
(\omega_{1}(\lambda^{k})+\tau\lambda^{k})+\frac{C_{2}C_{0}(2+\displaystyle\frac{\tau}{\Lambda})}{\lambda^{2}}
(\omega_{1}(\lambda^{k})+\omega_{2}(\lambda^{k})+\tau\lambda^{k})S_{k}+\frac{1}{4}S_{k}
\end{eqnarray*}
since $\displaystyle\frac{\tau C_{0}}{2\Lambda}\leq\frac{1}{4}$. To continue, the proof is similar to the {\bf Step 2} in the proof of Theorem $\ref{mr1}$, we omit it.
Therefore $\left\{S_{k}\right\}_{k=0}^{\infty}$ is bounded and consequently $\left\{S_{k}\right\}_{k=0}^{\infty}$ and $\sum\limits_{k=0}^{\infty} M_{k}$ are convergent. Moreover, from $(\ref{gk})$, $(\ref{fk})$ and $(\ref{ek})$, we know that $\left\{G_{k}\right\}_{k=0}^{\infty}$, $\left\{F_{k}\right\}_{k=0}^{\infty}$ and $\left\{E_{k}\right\}_{k=0}^{\infty}$ are bounded. It is also clear that $\lim\limits_{k\rightarrow\infty}M_{k}=0$.

\
\

{\bf Step 3:} Prove that $\lim\limits_{k \rightarrow+\infty} E_{k} $, $\lim\limits_{k \rightarrow+\infty} F_{k} $ and $\lim\limits_{k \rightarrow+\infty} G_{k} $ exist.

\
\

Actually, it is clear that
$$
\sum_{k=0}^{\infty}|G_{k+1}-G_{k}| \leq C_{0}\sum_{k=0}^{\infty}M_{k}<\infty.
$$
Therefore $\left\{G_{k}\right\}_{k=0}^{\infty}$ is a Cauchy sequence, and consequently $\lim\limits_{k \rightarrow+\infty}G_{k}$ exists. We set $\overline{G}:=$
$\lim\limits_{k \rightarrow+\infty} G_{k} $. Similarly, we can set $\overline{F}:=$
$\lim\limits_{k \rightarrow+\infty} F_{k} $, $\overline{E}:=$
$\lim\limits_{k \rightarrow+\infty} E_{k} $. Then we denote $\overline{P}(x)=\overline{E}+\overline{F}\cdot x+x^{T}\overline{G}x$.

\
\

{\bf Step 4:} Prove $u-v$ is $C^{2}$ at 0.

\
\

It is sufficient to show that there exists a sequence $\{N_{k}\}$ satisfying that $\lim\limits_{k \rightarrow \infty} N_{k}=0$ and
\begin{equation}\label{final}
\frac{\left\|u-v-\overline{P}+\vec{b}(0)\cdot\overline{F}\frac{x_{1}^{2}}{2a_{11}(0)}\right\|_{L^{\infty}(B_{\lambda^{k}})}}{\lambda^{2k}} \leq N_{k}.
\end{equation}
In fact, to show $(\ref{final})$, we notice that for any $k \geq 0$,
\begin{eqnarray*}
\frac{\left\|u-v-\overline{P}+\vec{b}(0)\cdot\overline{F}\frac{x_{1}^{2}}{2a_{11}(0)}\right\|_{L^{\infty}(B_{\lambda^{k}})}}{\lambda^{2k}}
&\leq&\frac{\left\|u-v-P_{k}+\vec{b}(0)\cdot F_{k}\frac{x_{1}^{2}}{2a_{11}(0)}\right\|_{L^{\infty}(B_{\lambda^{k}})}}{\lambda^{2k}}\\
&&+\frac{\left\|P_{k}-\overline{P}\right\|_{L^{\infty}(B_{\lambda^{k}})}}{\lambda^{2k}}
+\frac{\left\|\vec{b}(0)\cdot(F_{k}-\overline{F})\frac{x_{1}^{2}}{2a_{11}(0)}\right\|_{L^{\infty}(B_{\lambda^{k}})}}{\lambda^{2k}}.
\end{eqnarray*}
By the definition of $M_{k}$ we get
\begin{eqnarray*}
\frac{\left\|u-v-\overline{P}+\vec{b}(0)\cdot\overline{F}\frac{x_{1}^{2}}{2a_{11}(0)}\right\|_{L^{\infty}(B_{\lambda^{k}})}}{\lambda^{2k}}
&\leq& M_{k}+\frac{|E_{k}-\overline{E}|+|F_{k}-\overline{F}|\lambda^{k}+|G_{k}-\overline{G}|\lambda^{2k}}{\lambda^{2k}}
+\frac{\tau}{2\Lambda}|F_{k}-\overline{F}|\\
&\leq& M_{k}+|G_{k}-\overline{G}|+\frac{|F_{k}-\overline{F}|}{\lambda^{k}}+\frac{|E_{k}-\overline{E}|}{\lambda^{2k}}
+\frac{\tau}{2\Lambda}|F_{k}-\overline{F}|.
\end{eqnarray*}
We set $$N_{k}=M_{k}+|G_{k}-\overline{G}|+\displaystyle\frac{|F_{k}-\overline{F}|}{\lambda^{k}}+\displaystyle\frac{|E_{k}-\overline{E}|}{\lambda^{2k}}
+\displaystyle\frac{\tau}{2\Lambda}|F_{k}-\overline{F}|$$.
Then by using $(\ref{fk})$, we have
\begin{eqnarray*}
|F_{k}-\overline{F}|&\leq&|F_{k}-F_{k+1}|+|F_{k+1}-F_{k+2}|+\cdots\\
&\leq& C_{0}\left(\lambda^{k}M_{k}+\lambda^{k+1}M_{k+1}+\cdots\right)\\
&\leq& C_{0}\lambda^{k}\sum\limits_{i=k}^{\infty}M_{i}.
\end{eqnarray*}
It follows that
$$\frac{|F_{k}-\overline{F}|}{\lambda^{k}}\leq C_{0}\sum\limits_{i=k}^{\infty}M_{i}.
$$
Since $\sum\limits_{i=0}^{\infty}M_{i}$ is convergent, then we have $\lim\limits_{k\rightarrow\infty}\displaystyle\frac{|F_{k}-\overline{F}|}{\lambda^{k}}=0$. Similarly we have $\lim\limits_{k\rightarrow\infty}\displaystyle\frac{|E_{k}-\overline{E}|}{\lambda^{2k}}=0$. Consequently we obtain that $\lim\limits_{k\rightarrow\infty}N_{k}=0$.

\
\

{\bf Step 5:} Prove $u$ is $C^{1,1}$ at 0.

\
\

From {\bf Step 4}, we know that $u-v$ is $C^{2}$ at 0. Since $v$ is $C^{1,1}$ at 0, then $u$ is $C^{1,1}$ at 0.

\end{document}